\documentclass[11pt,naturalnames]{article}

\title{Density of growth rates of subgroups of a  free group -- an alternative proof}
\author{\'Ad\'am Tim\'ar}

\renewcommand\footnotemark{}
\usepackage[ruled,vlined]{algorithm2e}
\usepackage{amsmath}
\usepackage{amsthm}
\usepackage{amsfonts}
\usepackage{graphicx}
\newif\ifhyper\IfFileExists{hyperref.sty}{\hypertrue}{\hyperfalse}
\ifhyper\usepackage{hyperref}\fi
\usepackage{enumitem}
\usepackage{subfig}
\usepackage{caption}
\usepackage{keyval}
\usepackage[margin=1in]{geometry}
\usepackage{setspace}
\usepackage[displaymath, mathlines,pagewise]{lineno}
\theoremstyle{definition}

\newtheorem{theorem}{Theorem}
\newtheorem{lemma}[theorem]{Lemma}

\def\supp{{\rm supp}}
\def\max{{\rm max}}
\def\min{{\rm min}}
\def\dist{{\rm dist}}
\def\Aut{{\rm Aut}}
\def\id{{\rm id}}
\def\Stab{{\rm Stab}}

\begin{document}
\maketitle
\let\thefootnote\relax\footnotetext{\footnotesize{Partially supported by Icelandic  Research Fund grant 239736-051 and the ERC
grant No. 810115-DYNASNET.}}

\bigskip

\def\eref#1{(\ref{#1})}
\newcommand{\Prob} {{\bf P}}
\newcommand{\C}{\mathcal{C}}
\newcommand{\LL}{\mathcal{L}}
\newcommand{\Z}{\mathbb{Z}}
\newcommand{\N}{\mathbb{N}}
\newcommand{\HH}{\mathbb{H}}
\newcommand{\Rr}{\mathbb{R}^3}
\newcommand{\h}{\mathcal{H}}
\def\diam{\mathrm{diam}}
\def\length{\mathrm{length}}
\def\ev#1{\mathcal{#1}}
\def\Isom{{\rm Isom}}
\def\Re{{\rm Re}}
\def \eps {\epsilon}
\def \P {{\Bbb P}}
\def \E {{\Bbb E}}

\def\supp{{\rm supp}}
\def\max{{\rm max}}
\def\min{{\rm min}}
\def\dist{{\rm dist}}
\def\Aut{{\rm Aut}}
\def\id{{\rm id}}
\def\Stab{{\rm Stab}}

\newcommand{\lra}{\leftrightarrow}
\newcommand{\xlra}{\xleftrightarrow}
\newcommand{\xnlra}{\xnleftrightarrow}
\newcommand{\pc}{{p_c}}
\newcommand{\pt}{{p_T}}
\newcommand{\ptk}{{\hat{p}_T}}
\newcommand{\pl}{{\tilde{p}_c}}
\newcommand{\pe}{{\hat{p}_c}}
\newcommand{\pr}{\mathrm{\mathbb{P}}}
\newcommand{\pp}{\mu}
\newcommand{\ex}{\mathrm{\mathbb{E}}}
\newcommand{\ee}{\mathrm{\overline{\mathbb{E}}}}

\newcommand{\om}{{\omega}}
\newcommand{\ebd}{\partial_E}
\newcommand{\ivbd}{\partial_V^\mathrm{in}}
\newcommand{\ovbd}{\partial_V^\mathrm{out}}
\newcommand{\q}{q}
\newcommand{\TT}{\mathfrak{T}}
\newcommand{\T}{\mathcal{T}}
\newcommand{\RR}{\mathcal{R}}

\newcommand{\CC}{\Pi}
\newcommand{\BB}{\Pi}

\newcommand{\A}{\mathcal{A}}
\newcommand{\cc}{\mathbf{c}}
\newcommand{\pa}{{PA}}
\newcommand{\degi}{\deg^{in}}
\newcommand{\dego}{\deg^{out}}
\def\Pn{{\bf P}_n}

\newcommand{\R}{\mathbb R}
\newcommand{\F}{F}
\newcommand{\FF}{\mathfrak{F}}
\newcommand{\Can}{\rm Can}
\newcommand{\Vol}{\mathrm Vol}
\def\br{{\rm br}}
\def\gr{{\rm gr}}
\def\Gstar{{{\cal G}_{*}}}
\def\Gstarstar{{{\cal G}_{**}}}
\def\Gstarplus{{{\cal G}_{*}^\frown}}
\def\Rel{{\cal R}}
\def\Comp{{\rm Comp}}
\def\calT{{\cal T}}
\def\omps{{\omega_\delta^\eps}}

\begin{abstract}
We give an alternative proof to the theorem recently proved by Louvaris, Wise and Yehuda, 
that the growth rates of finitely generated subgroups of $F_r$ are dense in $[1,2r-1]$.
\end{abstract}

\bigskip

%\begin{theorem}\label{main}

%\end{theorem}

\section{Introduction}

%Let $G$ be a finitely generated group with finite generating set $S$. The exponential growth rate of $G$ with respect to $S$ is defined by
%\[
%\omega(G,S) = \lim_{n\to\infty} \left|\{ g \in G : |g|_S \le n \}\right|^{1/n},
%\]
%when the limit exists. This invariant is independent of the choice of $S$ up to equivalence and provides a coarse measure of the asymptotic size of $G$.

%In the case of a free group $F_r$ of rank $r \ge 2$, it is natural to consider the growth of its subgroups measured with the ambient word metric. 
Fix a free generating set $S$ for the free group $F_r$ of rank $r \ge 2$. For a subgroup $H \le F_r$, define the growth function
\[
\gamma_H(n) = \bigl| \{ h \in H : |h|_S \le n \} \bigr|,
\]
where $|\cdot|_S$ denotes word length in $F_r$ with respect to $S$, and the (exponential) growth rate of $H$ by
\[
\omega(H) = \lim_{n \to \infty} \gamma_H(n)^{1/n}.
\]
This limit exists by Fekete's lemma.
%This definition depends on the choice of the generating set $S$, but any two free bases of $F_r$ yield equivalent notions of growth, so that $\omega(H)$ is well-defined up to this natural equivalence.

The growth rate of $F_r$ itself with respect to a free basis equals $2r-1$. Any subgroup $H \le F_r$ has growth rate lying in the interval $[1,2r-1]$. Understanding which values in this interval arise as growth rates of subgroups of $F_r$ is a natural problem, closely linked to the correspondence between subgroups and graph immersions and to the spectral theory of non-backtracking operators. See \cite{LWY2} for an overview of related results, including those about groups where the growth rates are not dense (that is, a ``growth-gap'' is present), and \cite{KP} for results on the possible values of several monotone parameters (e.g., spectral radius, asymptotic entropy, percolation thresholds) associated with Cayley graphs of finitely generated groups.

\begin{theorem}[Louvaris-Wise-Yehuda, \cite{LWY2}]\label{main}
The set of growth rates of finitely generated subgroups of $F_r$ is dense in $[1,2r-1]$.
\end{theorem} 
%These growth rates are algebraic numbers, arising as spectral radii of non-backtracking matrices associated to finite graph immersions. 
The original proof is quite involved, using probabilistic constructions of graphs with prescribed spectral behavior and concentration results for the leading eigenvalue.

%In subsequent work, a complementary result was obtained: every real number $\alpha \in [1,2r-1]$ occurs as the growth rate of some subgroup of $F_r$. In contrast to the density result for finitely generated subgroups, this realization theorem allows infinitely generated subgroups, which are constructed as limits of finite immersions whose growth rates converge to $\alpha$.

The aim of the present paper is to provide a short and elementary proof of the theorem. 
%Our arguments emphasize direct constructions and simple spectral estimates, and avoid the probabilistic machinery previously employed. This yields a streamlined approach to the description of possible subgroup growth rates in free groups.

%The main theorem is reduced to the  following key statement by a simple argument (see the proof of Theorem 1.1 in \cite{LWY2}). Denote by $\lambda_1(B_G)$ the leading eigenvalue of the nonbacktracking matrix of $G$. For the definition we refer the reader to \cite{LWY2}; what will matter to us is that this quantity is equal to the growth of the universal cover of $G$.

%For a graph $G$ of degrees at most $2r$ one can label the edges by  
%The connection between $\lambda_1(G)$ and growth rates of subgroups in $F_r$ is as follows. The growth rate of a subgroup $\Gamma$ in $F_r$ is the same as $\omega (G):=\lim_{n\to\infty} {\rm NB}_v (n)^{1/n}$, where ${\rm NB}_v (n)$ is the number of non-backtracking closed walks from basepoint $v$ in $G$. 

Given a finite connected graph $G$ of degrees in $\{2,\ldots, 2r\}$, one can turn it into a Schreier graph of $F_r$ (see e.g. \cite{C}) by choosing a basepoint, adding loops to make it $2r$-regular, and then suitably orienting the edges and labelling them by generators. Then the closed non-backtracking walks on this graph will correspond to the elements of a subgroup in $F_r$, hence the growth of this subgroup in the word metric of $F_r$ is given by the leading eigenvalue (Perron eigenvalue) $\lambda_1(B_G)$ of the non-backtracking matrix of $G$.
This simple reduction shows that to prove Theorem \ref{main}, it is enough to find a family of finite connected graphs of degrees in $\{2,\ldots, 2r\}$ such that the leading eigenvalue of their non-backtracking matrix is dense in $[1,2r-1]$. This is Theorem 3.2 in \cite{LWY2}. 

\begin{theorem}[Louvaris-Wise-Yehuda, \cite{LWY2}]\label{graph}
For every $1<r\in\N$ and $\alpha\in(1,2r-1)$ and $\eps>0$ there is a finite graph $G$ with degrees in $\{2,\ldots,2r\}$ such that $|\lambda_1(B_G)-\alpha|<\eps$.
\end{theorem}

%In \cite{LWY2} it was shown that for finitely generated subgroups of a free group $F$ of rank $r$, the growth rates with respect to $F$ form a dense subset of $[1,2r-1]$. A short argument reduces this problem to finding a family of finite connected graphs of degrees in $\{2,\ldots, 2r\}$ such that the leading eigenvalue of their non-backtracking matrix is dense in $[1,2r-1]$ (see Theorem 3.2 in \cite{LWY2}). 

%The growth rate of the universal cover of a connected finite graph $G$ is equal to $\lambda_1 (B_G)$, where $B_G$ is the non-backtracking matrix of $G$ and $\lambda_1 (B_G)$ is its leading eigenvalue.  See Lemma 6 in \cite{EH}.

Call a tree $T$ {\it strongly periodic} if it is the universal cover of a finite simple graph. The growth rate $\lim_{n\to\infty}|B_n(T)|^{1/n}=\lim_{n\to\infty}|S_n(T)|^{1/n}$ of such a tree always exists, where $S_n(T)$ and $B_n(T)$ respectively denote the sphere and ball of radius $n$ around a fixed vertex.
The leading eigenvalue of the non-backtracking matrix of $G$ is the same as the growths of the universal cover of $G$ (see e.g. Lemma 6 in \cite{EH}). Hence the theorem is equivalent to saying that the set of growths of strongly periodic trees with degrees in $\{2,\ldots,2r\}$ is dense in $[1,2r-1]$.

\section{Proofs}

\begin{lemma}
Let $n\in \N$, and $T$ be an arbitrary strongly periodic tree of degree at most $d$. Let $F\subset E(T)$ be a subset of edges of pairwise distance at least $n+1$. Construct $T'$ from $T$ by subdividing every edge of $F$ (i.e., replacing it by a path of length 2). Then if $T'$ is strongly periodic then $\gr (T')\geq \gr (T)^{\frac{n}{n+1}}$, and hence
$$0\leq \gr (T)-\gr (T')\leq \gr (T)(1-\gr(T)^{\frac{-1}{n+1}}).
$$
\end{lemma}

\begin{proof}
Consider the vertex set $V(T')$ as the union of $V(T)$ and the new vertices used in the subdivision of edges.
For every $t\in\N$ we have $V(S_{tn}(T))
\subset V(S_{t(n+1)}(T'))$, hence
$$\gr(T')=\lim_{t\to\infty} |S_{t(n+1)}(T')|^{\frac{1}{t(n+1)}}\geq \lim_{t\to\infty} \bigl(|S_{tn}(T)|^{\frac{1}{tn}}\bigr)^{\frac{tn}{t(n+1)}}=\gr(T)^{\frac{n}{n+1}}.
$$

The inequality $0\leq \gr (T)-\gr (T')$ follows by direct computation, or from the fact that we can retain $T$ from $T'$ by contracting one edge for each of the subdivided edges, and the critical percolation probability $p_c=\frac{1}{\br}=\frac{1}{\gr}$ is monotone decreasing under contractions of edges whenever every contracted vertex has a finite set of preimages. (Note that the branching number $\br$ always equals $p_c^{-1}$, and it is equal to $\gr$ for periodic trees such as $T$ and $T'$, see \cite{LP}.)
\end{proof}

\begin{proof}[Proof of Theorem~\ref{graph}]
As noted after Theorem \ref{graph}, the claim is equivalent to the growths of all strongly periodic trees being dense. It is enough to show that 
for every $K\in \N$ and $(a,a+\eps)\subset \bigl((d-1)^{\frac{1}{2K}},(d-1)^{\frac{1}{K}}\bigr)$, $\eps>0$, there exists a strongly periodic tree of growth in $(a,a+\eps)$. 

Fix $n\in\N$ such that 
$$(d-1)^{\frac{1}{K}}-(d-1)^{\frac{n}{K(n+1)}}=(d-1)^{\frac{1}{K}}(1-(d-1)^{\frac{-1}{K(n+1)}})<\eps.$$

Replacing each edge of the $d$-regular tree by a path of length $K$ results in a strongly periodic tree $\T_K$ that has growth $(d-1)^{\frac{1}{K}}$. 
If $H$ is a finite graph such that $\T_K$ is the universal cover of $H$, first take a lift $H'$ of $H$ that has girth at least $n+1$. (Such a lift exists, see \cite{1}.) Clearly, $\T_K$ is also a universal cover for $H'$. Now, for a suitable $N$, assign to each edge of $H'$ a distinct color of the color set $\{1,\ldots,N\}$, and pull this coloring back by the cover map to obtain a coloring of the edges of $\T_K$ by $\{1,\ldots,N\}$ with the property that edges of the same color are at distance at least $n+1$ from each other. Starting with $T^0:=\T_K$, recursively define $T^{i}$ from $T^{i-1}$ by subdividing each of its edges of color $i$. Every $T^i$ is strongly periodic, because it is the universal cover of the finite graph obtained from $H'$ by subdividing its edges of colors $\{1,\ldots,i\}$. By the lemma, 
$$
(d-1)^{\frac{1}{K}}=\gr (\T_K)=\gr(T^0)\geq\gr(T^1)\geq\ldots\geq\gr(T^{N-1})\geq\gr(T^N)= (d-1)^{\frac{1}{2K}},
$$
and
$$\gr(T^i)-\gr(T^{i+1})\leq \gr (T^i)(1-\gr(T^i)^{\frac{-1}{n+1}})\leq \gr (T^0)(1-\gr(T^0)^{\frac{-1}{n+1}}) \leq (d-1)^{\frac{1}{K}}(1-(d-1)^{\frac{-1}{K(n+1)}})<\eps,
$$
where the second inequality follows from the fact that $x(1-x^{\frac{-1}{n+1}})$ is monotone increasing on $[1,\infty)$. Hence as $\gr(T^i)$ proceeds from $(d-1)^{\frac{1}{K}}$ to $(d-1)^{\frac{1}{2K}}$, always changing by less than $\eps$, one of the $\gr(T^i)$ values has to be in $(a,a+\eps)$. 
\end{proof}
Theorem \ref{main} now follows, by the argument before Theorem \ref{graph}.

%\noindent
%{\bf Acknowledgments:} 

\bigskip

\ \\
Division of Mathematics, The Science Institute, University of Iceland, Reykjavik, Iceland,\\
and\\
HUN-REN Alfr\'ed R\'enyi Institute of Mathematics, Budapest, Hungary.

\noindent
\texttt{madaramit[at]gmail.com}\\ 

%{\bf \'Ad\'am Tim\'ar}\\
%Division of Mathematics, The Science Institute, University of Iceland\\
%Dunhaga 3 IS-107 Reykjavik, Iceland\\
%and\\
%Alfr\'ed R\'enyi Institute of Mathematics\\
%Re\'altanoda u. 13-15, Budapest 1053 Hungary\\
%\texttt{madaramit[at]gmail.com}\\

\end{document}